\documentclass[oneside,a4paper,12pt,leqno]{amsart}

\usepackage{verbatim,amsxtra,amssymb,amsmath,color,psfrag,graphicx}

\definecolor{mygrey}{gray}{0.70}
\definecolor{mygreen}{rgb}{0,.75,0}
\definecolor{myred}{rgb}{1,0,0}
\definecolor{orange}{rgb}{1,.5,0}

\numberwithin{equation}{section}

\newtheorem{thm}[equation]{Theorem}
\newtheorem{claim}[equation]{Claim}

\newtheorem{cor}[equation]{Corollary}
\newtheorem{lem}[equation]{Lemma}
\newtheorem{prop}[equation]{Proposition}

\newtheorem*{thrm}{Main Theorem}
\newtheorem*{thrm1}{\thmref{th:dim}}
\newtheorem*{thrm2}{\thmref{th:estimate}}

\theoremstyle{definition}
\newtheorem{defn}[equation]{Definition}

\theoremstyle{remark}
\newtheorem{rem}[equation]{Remark}

\newcommand{\thmref}[1]{Theorem~\ref{#1}}
\newcommand{\propref}[1]{Proposition~\ref{#1}}
\newcommand{\lemref}[1]{Lemma~\ref{#1}}

\newcommand{\secref}[1]{Section~\ref{#1}}

\newcommand{\claimref}[1]{Claim~\ref{#1}}

\newcommand{\conv}[1]
{\vskip -.3cm $$ {\vbox{\vrule width 1mm \hskip 4mm \hbox{\parbox{14cm}{ {#1}
}}}\hfill} $$}

\renewcommand\a{\alpha}
\newcommand\af{\mathfrak{a}}
\newcommand\an{\measuredangle}

\newcommand\B{\mathcal{B}}
\renewcommand\b{\beta}
\newcommand\bnd{\operatorname{\textbf{bnd}}}

\newcommand\C{\mathbb C}
\newcommand\ch{\check}

\renewcommand\d{\delta}
\newcommand\diag{\operatorname{diag}}
\newcommand\diam{\operatorname{diam}}

\newcommand\E{\mathcal{E}}
\newcommand\e{\varepsilon}
\newcommand\emp{\varnothing}
\newcommand\esssup{\operatorname{ess\,sup}}

\newcommand\G{\Gamma}
\newcommand\g{\gamma}

\newcommand\Gr{\mathcal{G}}

\newcommand\h{\boldsymbol{h}}

\renewcommand\l{\lambda}
\newcommand\la{\langle}

\newcommand\ov{\overline}

\renewcommand\P{\mathbf{P}}
\renewcommand\Pr{\mathcal{P}}
\newcommand\pt{\partial}

\newcommand\R{\mathbb{R}}
\renewcommand\r{\rho}
\newcommand\ra{\rangle}

\newcommand\spn{\operatorname{span}}
\newcommand\supp{\operatorname{supp}}

\renewcommand\th{\theta}
\newcommand{\toto}{\mathop{\,\longrightarrow\,}}

\newcommand\un{\underline}

\newcommand\V{\mathcal{V}}

\newcommand\X{\mathfrak{X}}
\newcommand\x{{\boldsymbol{x}}}

\newcommand\Z{\mathbb{Z}}
\newcommand\z{\zeta}

\begin{document}

\title{Matrix random products with singular harmonic measure}

\author{Vadim A. Kaimanovich}
\address{Mathematics, Jacobs University Bremen, Campus Ring 1, D-28759, Bremen, Germany}
\email{v.kaimanovich@jacobs-university.de}

\author{Vincent Le Prince}
\address{IRMAR, Campus de Beaulieu, 35042 Rennes, France}
\email{vincent.leprince@univ-rennes1.fr}



\begin{abstract}
Any Zariski dense countable subgroup of $SL(d,\R)$ is shown to carry a
non-degenerate finitely supported symmetric random walk such that its harmonic
measure on the flag space is singular. The main ingredients of the proof are:
(1) a new upper estimate for the Hausdorff dimension of the projections of the
harmonic measure onto Grassmannians in $\R^d$ in terms of the associated
differential entropies and differences between the Lyapunov exponents; (2) an
explicit construction of random walks with uniformly bounded entropy and
Lyapunov exponents going to infinity.
\end{abstract}

\maketitle

\thispagestyle{empty}

\section*{Introduction}

The notion of \emph{harmonic measure} (historically first defined in the
framework of the theory of potential) has an explicit probabilistic
description in terms of the dynamical properties of the associated Markov
processes as a \emph{hitting distribution}. Moreover, ``hitting'' can be
interpreted both as attaining the target set in the usual sense (in finite
time) and as converging to it at infinity (when the target is attached as a
boundary to the state space).

In concrete situations the target set is usually endowed with additional
structures giving rise to other ``natural'' measures (e.g., smooth, uniform,
Haar, Hausdorff, maximal entropy, etc.), which leads to the question about
\emph{coincidence} of the harmonic and these ``other'' measures (or, in a
somewhat weaker form, about coincidence of the respective measure classes). As
a general rule, such a coincidence either implies that the considered system
has very special symmetry properties or is not possible at all. However,
establishing it in a rigorous way is a notoriously difficult problem. See, for
example, the cases of the Brownian motion on cocompact negatively curved
Riemannian manifolds \cite{Katok88,Ledrappier95}, of Julia sets of
endomorphisms of the Riemann sphere
\cite{Przytycki-Urbanski-Zdunik89,Zdunik91} and of polynomial-like maps
\cite{Lyubich-Volberg95, Balogh-Popovici-Volberg97, Zdunik97}, or of Cantor
repellers in a Euclidean space \cite{Makarov-Volberg86,Volberg93}.

\medskip

In the present paper we consider the singularity problem for \emph{random
matrix products} $x_n=h_1h_2\dots h_n$ with Bernoulli increments, or, in other
words, for \emph{random walks} on the group $SL(d,\R)$ (or its subgroups).
Actually, our results are also valid for general non-compact semi-simple Lie
groups, see \cite{LePrince04}, but for the sake of expositional simplicity we
restrict ourselves just to the case of matrix groups. Another simplification
is that we always assume that the considered subgroups of $SL(d,\R)$ are
Zariski dense. It guarantees that the harmonic measure of the random walk is
concentrated on the space of full flags in $\R^d$ (rather than on its quotient
determined by the degeneracies of the Lyapunov spectrum). Modulo an
appropriate technical modification our results remain valid without this
assumption as well.

If the distribution $\mu$ of the increments $\{h_n\}$ has a finite \emph{first
moment} $\int\log\|h\|\,d\mu(h)$, then by the famous \emph{Oseledets
multiplicative ergodic theorem} \cite{Oseledec68} there exists the
\emph{Lyapunov spectrum} $\l$ consisting of \emph{Lyapunov exponents}
$\l_1\ge\l_2\ge\dots\l_d$ (they determine the growth of the random products in
various directions), and, moreover, a.e.\ sample path $\{x_n\}$ gives rise to
the associated \emph{Lyapunov flag} (filtration) of subspaces in $\R^d$. The
distribution $\nu=\nu(\mu)$ of these Lyapunov flags is then naturally called
the \emph{harmonic measure} of the random product. The geometric
interpretation of the Oseledets theorem \cite{Kaimanovich89} is that the
sequence $x_n o$ asymptotically follows a geodesic in the associated
Riemannian symmetric space $S=SL(d,\R)/SO(d)$ (here $o=SO(d)\in S$); the
Lyapunov spectrum determines the Cartan (``radial'') part of this geodesic,
whereas the Lyapunov flag determines its direction.

If $\supp\mu$ generates a Zariski dense subgroup of $SL(d,\R)$, then the
Lyapunov spectrum is \emph{simple} ($\equiv$ the vector $\l$ lies in the
interior of the positive Weyl chamber), see
\cite{Guivarch-Raugi85,Goldsheid-Margulis89}, so that the associated Lyapunov
flags are full ($\equiv$ contain subspaces of all the intermediate
dimensions). The space $\B=\B(d)$ of full flags in $\R^d$ is also known under
the name of the \emph{Furstenberg boundary} of the associated symmetric space
$S$, see \cite{Furstenberg63} for its definition and
\cite{Kaimanovich89,Guivarch-Ji-Taylor98} for its relation with the boundaries
of various compactifications of Riemannian symmetric spaces. In the case when
the Lyapunov spectrum is simple, a.e.\ sequence $x_n o$ is convergent in all
reasonable compactifications of the symmetric space $S$, and the corresponding
hitting distributions can be identified with the harmonic measure $\nu$ on
$\B$.

\medskip

The flag space $\B$ is endowed with a natural smooth structure, therefore it
makes sense to compare the harmonic measure class with the smooth measure
class (the latter class contains the unique rotation invariant measure on the
flag space). The harmonic measure is ergodic, so that it is either absolutely
continuous or singular with respect to the smooth (or any other
quasi-invariant) measure class. Accordingly, we shall call the jump
distribution $\mu$ either \emph{absolutely continuous} or \emph{singular at
infinity}.

\medskip

If the measure $\mu$ is absolutely continuous with respect to the Haar measure
on the group $SL(d,\R)$ (or even weaker: a certain convolution power of $\mu$
contains an absolutely continuous component) then it is absolutely continuous
at infinity \cite{Furstenberg63,Azencott70}.

As it turns out, there are also measures $\mu$ which are absolutely continuous
at infinity in spite of being supported by a \emph{discrete subgroup} of
$SL(d,\R)$. Namely, Furstenberg \cite{Furstenberg71,Furstenberg73} showed that
any lattice (for instance, $SL(d,\Z)$) carries a probability measure $\mu$
with a finite first moment which is absolutely continuous at infinity. It was
used for proving one of the first results on the rigidity of lattices in
semi-simple Lie groups.

Furstenberg's construction of measures absolutely continuous at infinity
(based on discretization of the Brownian motion on the associated symmetric
space) was further extended and generalized in
\cite{Lyons-Sullivan84,Kaimanovich92a,Ballmann-Ledrappier96}. Another
construction of random walks with a given harmonic measure was recently
developed by Connell and Muchnik \cite{Connell-Muchnik07,Connell-Muchnik07a}.
Note that the measures $\mu$ arising from all these constructions are
inherently infinitely supported.

\medskip

Let us now look at the singularity vs. absolute continuity dichotomy for the
harmonic measure from the ``singularity end''. The first result of this kind
was obtained by Chatterji \cite{Chatterji66} who established singularity of
the distribution of infinite continuous fractions with independent digits.
This distribution can indeed be viewed as the harmonic measure associated to a
certain random walk on $SL(2,\Z)\subset SL(2,\R)$ \cite{Furstenberg63a}. See
\cite{Chassaing-Letac-Mora83} for an explicit description of the harmonic
measure in a similar situation and \cite{Kifer-Peres-Weiss01} for a recent
very general result on singularity of distributions of infinite continuous
fractions. Extending the continuous fractions setup Guivarc'h and Le Jan
\cite{Guivarch-LeJan90,Guivarch-LeJan93} proved that the measures $\mu$ on
non-compact lattices in $SL(2,\R)$ satisfying a certain moment condition (in
particular, all finitely supported measures) are singular at infinity.

Together with some other circumstantial evidence (as, for instance, pairwise
singularity of various natural boundary measure classes on the visibility
boundary of the universal cover of compact negatively curved manidolds
\cite{Ledrappier95} or on the hyperbolic boundary of free products
\cite{Kaimanovich-LePrince08p}) the aforementioned results lead to the
following

\medskip

{\parindent 0pt \textsc{Conjecture.} Any finitely supported probability
measure on $SL(d,\R)$ is singular at infinity. }

\medskip

The principal result of the present paper is the following

\begin{thrm} \label{th:mn}
Any countable Zariski dense subgroup $\G$ of $SL(d,\R)$ carries a
non-degenerate symmetric probability measure $\mu$ which is singular at
infinity. Moreover, if $\G$ is finitely generated then the measure $\mu$ can
be chosen to be finitely supported.
\end{thrm}

Note here an important difference between the case $d=2$ and the higher rank
case $d\ge 3$. If $\G$ is discrete, then for $d=2$ the boundary circle can be
endowed with different $\G$-invariant smooth structures (parameterized by the
points from the Teichm\"uller space of the quotient surface), so that
Furstenberg's discretization construction readily provides existence of
measures $\mu$ (not finitely supported though!) which are singular at
infinity. Obviously, this approach does not work in the higher rank case,
where our work provides first examples of measures singular at infinity.

\medskip

Our principal tool for establishing singularity is the notion of the
\emph{dimension} of a measure. Namely, we show that in the setup of our Main
Theorem the measure $\mu$ can be chosen in such a way that the Hausdorff
dimension of the associated harmonic measure (more precisely, of its
projection onto one of the Grassmanians in $\R^d$) is arbitrarily small, which
obviously implies singularity. For doing this we establish an inequality
connecting the \emph{Hausdorff dimension} with the \emph{asymptotic entropy}
and the \emph{Lyapunov exponents} of the random walk.

\medskip

There are several notions of dimension of a probability measure $m$ on a
compact metric space $(Z,\r)$ (see \cite{Pesin97} and the discussion in
\secref{sec:Hausd} below). These notions roughly fall into two categories. The
\emph{global} ones are obtained by looking at the dimension of sets which
``almost'' (up to a piece of small measure $m$) coincide with $Z$; in
particular, the \emph{Hausdorff dimension} $\dim_H m$ of the measure $m$ is
$\inf\left\{ \dim_H A : m(A)=1 \right\}$, where $\dim_H A$ denotes the
Hausdorff dimension of a subset $A\subset Z$. On the other hand, the
\emph{local} ones are related to the asymptotic behavior of the measures of
concentric balls $B(z,r)$ in $Z$ as the radius $r$ tends to $0$. More
precisely, the lower $\un\dim_P m(z)$ (resp., the upper $\ov\dim_P m(z)$)
\emph{pointwise dimension} of the measure $m$ at a point $z\in Z$ is defined
as the $\liminf$ (resp., $\limsup$) of the ratio $\log m B(z,r)/\log r$ as
$r\to 0$. In these terms $\dim_H m$ coincides with $\esssup_z \un\dim_P m(z)$.
In particular, if $\un\dim_P m(z)=\ov\dim_P m(z)=D$ almost everywhere for a
constant $D$, then $\dim_H m=D$. Moreover, in the latter case all the
reasonable definitions of dimension of the measure $m$ coincide.

\medskip

Numerous variants of the formula $\dim m = h/\l$ relating the dimension of an
invariant measure $m$ of a differentiable map with its entropy $h=h(m)$ and
the characteristic exponent(s) $\l=\l(m)$ have been known since the late 70s
-- early 80s, see \cite{Ledrappier81,Young82} and the references therein.
Ledrappier \cite{Ledrappier83} was the first to carry it over to the context
of random walks by establishing the formula $\dim \nu =h/2\l$ for the
dimension of the harmonic measure of random walks on discrete subgroups of
$SL(2,\C)$. Here $h=h(\mu)$ is the asymptotic entropy of the random walk with
the jump distribution $\mu$, and $\l=\l(\mu)$ is the Lyapunov exponent.
However, the dimension appearing in this formula is somewhat different from
the classical Hausdorff dimension, because convergence of the ratios $\log m
B(z,r)/\log r$ to the value of dimension is only established in measure
(rather than almost surely).

Le Prince has recently extended the approach of Ledrappier (also see
\cite{Ledrappier01,Kaimanovich98}) to a general discrete group $G$ of
isometries of a Gromov hyperbolic space $X$; in this case the Gromov product
induces a natural metric (more rigorously, a gauge) on the hyperbolic boundary
$\pt X$. He proved that if a probability measure $\mu$ on $G$ has a finite
first moment with respect to the metric on $X$, then the associated harmonic
measure $\nu$ on $\pt X$ has the property that $\ov\dim_P\nu(z)\le h/\ell$ for
$\nu$-a.e. point $z\in\pt X$ (which implies that $\dim_H\nu\le h/\ell$)
\cite{LePrince07}, and that the \emph{box dimension} of the measure $\nu$ is
precisely $h/\ell$ \cite{LePrince08} (here, as before, $h=h(\mu)$ is the
asymptotic entropy of the random walk $(G,\mu)$, and $\ell=\ell(\mu)$ is its
linear rate of escape with respect to the hyperbolic metric). Note that the
question about the pointwise dimension, i.e., about the asymptotic behaviour
of the ratios $\log m B(z,r)/\log r$ for a.e. point $z\in\pt X$, is still open
in this generality. It was recently proved that these ratios converge to
$h/\ell$ almost surely for any symmetric finitely supported random walk on $G$
\cite{Blachere-Haissinsky-Mathieu08}. This should also be true for finitely
supported random walks which are not necessarily symmetric. Namely, the
results of Ancona \cite{Ancona88} on almost multiplicativity of the Green
function imply that in this situation the harmonic measure has the
\emph{doubling property }, which in turn can be used to prove existence of the
pointwise dimension.

Yet another example is the inequality $\dim_H\nu\le h/\ell$ for the Hausdorff
dimension of the harmonic measure of iterated function systems in the
Euclidean space \cite{Nicol-Sidorov-Broomhead02} (when is the equality
attained?).

In the context of random walks on general countable groups the asymptotic
entropy $h(\mu)$, the rate of escape $\ell(\mu)$ and the exponential growth
rate of the group $v$ satisfy the inequality $h(\mu)\le\ell(\mu)v$ (it was
first established by Guivarc'h \cite{Guivarch79}; recently Vershik
\cite{Vershik00} revitalized interest in it). As it was explained above, in
the ``hyperbolic'' situations the ratio $h/\ell$ can (under various additional
conditions) be interpreted as the dimension of the harmonic measure, whereas
$v$ is the Hausdorff dimension of the boundary itself.

\medskip

Unfortunately, the aforementioned formulas of the type $\dim\nu=h/\ell$ can
not be directly carried over to the higher rank case. The reason for this is
that in this case the boundary is ``assembled'' of several components which
may have different dimension properties. The simplest illustration is provided
by the product $\G_1\times\G_2$ of two discrete subgroups of $SL(2,\R)$ with a
product jump distribution $\mu=\mu_1\otimes\mu_2$. The Furstenberg boundary of
the bidisk associated with the group $SL(2,\R)\times SL(2,\R)$ is the product
of the boundary circles of each hyperbolic disc. The harmonic measure of the
jump distribution $\mu$ is then the product of the harmonic measures of the
jump distributions $\mu_1$ and $\mu_2$ which may have different dimensions (it
would be interesting to investigate the dimension properties of the harmonic
measure on the boundary of the polydisc for groups and random walks which do
not split into a direct product). In the case of $SL(d,\R)$ the role of
elementary ``building blocks'' of the flag space $\B$ is played by the
Grassmannians $\Gr_i$ in $\R^d$ (which are the minimal equivariant quotients
of $\B$).

In order to deal with spaces with a $\mu$-stationary measure other than the
Poisson boundary of the random walk $(\G,\mu)$ it is convenient to use the
notion of the \emph{differential $\mu$-entropy} first introduced by
Furstenberg \cite{Furstenberg71}. It measures the ``amount of information''
about the random walk contained in the action. The differential entropy does
not exceed the asymptotic entropy of the random walk, coincides with it for
the Poisson boundary, and is strictly smaller for any proper quotient of the
Poisson boundary \cite{Kaimanovich-Vershik83}. Our point is that in the
formulas of the type $\dim\nu=h/\ell$ (see above) for spaces with a
$\mu$-stationary measure $\nu$ other than the Poisson boundary, the asymptotic
entropy $h$ should be replaced with the differential entropy of the space.

In the setup of our Main Theorem let $\mu$ be a non-degenerate probability
measure on $\G$ with a finite first moment, so that the associated Lyapunov
spectrum $\l_1>\l_2>\dots>\l_d$ is simple. Denote by $\nu_i$ the image of the
harmonic measure $\nu$ on the flag space $\B$ under the projection onto the
rank $i$ Grassmannian $\Gr_i$, and let $E_i$ be the corresponding differential
entropy. All the spaces $(\Gr_i,\nu_i)$ are quotients of the Poisson boundary
of the random walk $(\G,\mu)$; if $\G$ is discrete then $(\B,\nu)$ is the
Poisson boundary of the random walk $(\G,\mu)$
\cite{Kaimanovich85,Ledrappier85}, otherwise the Poisson boundary of the
random walk $(\G,\mu)$ may be bigger than the flag space
\cite{Kaimanovich-Vershik83,Bader-Shalom06,Brofferio06} (note that it is still
unknown whether all the proper quotients of the flag space $\B$ are also
always proper measure-theoretically with respect to the corresponding
$\mu$-stationary measures).

\begin{thrm1}
The Hausdorff dimensions of the measures $\nu_i$ satisfy the inequalities
$$
\dim \nu_i \le \frac{E_i}{\l_i-\l_{i+1}} \;.
$$
\end{thrm1}

In the case when $\G$ is a discrete subgroup of $SL(2,\R)$ the right-hand side
of the above inequality precisely coincides with the ratio $h(\mu)/2\l(\mu)$
from Ledrappier's formula \cite{Ledrappier83}.

\medskip

Our Main Theorem follows from a combination of this dimension estimate with
the following construction:

\begin{thrm2}
Let $\mu$ be a non-degenerate probability measure on $\G$ whose entropy and
the first moment are finite, and let $\g\in\G$ be an $\R$-regular element
(which always exists in a Zariski dense group). Then the measures
$$
\mu^k = \frac12\mu+\frac14\left( \d_{\g^k}+\d_{\g^{-k}} \right)
$$
have the property that the entropies $H(\mu^k)$ (and therefore the
corresponding asymptotic entropies as well) are uniformly bounded, whereas the
lengths of their Lyapunov vectors $\l(\mu^k)$ (equivalently, the top Lyapunov
exponents $\l_1(\mu^k)$) go to infinity.
\end{thrm2}

\medskip

Let us mention here several issues naturally arising in connection with our
results.

\medskip

1. Our study of dimension in the present paper is subordinate to proving
singularity at infinity, so that we only use rather rudimentary facts about
the dimension of the harmonic measure. Under what conditions is the inequality
from \thmref{th:dim} realized as equality? Is there an exact formula for the
dimension of the harmonic measure on the flag space? Of course, in these
questions one has to specify precisely which definition of the dimension is
used, cf. \secref{sec:Hausd}.

\medskip

2. For our purposes it was enough to establish in \thmref{th:estimate} that
the lengths of the Lyapunov vectors $\l(\mu^k)$ go to infinity. Is this also
true for all the spectral gaps $\l_i(\mu^k)-\l_{i+1}(\mu^k)$ (which would
imply that the dimensions of the harmonic measures on $\B$ go to 0)? What
happens with the harmonic measures $\nu^k$ (or with their projection); do they
weakly converge, and if so what is the limit? Note that one can ask the latter
questions in the hyperbolic (rank 1) situation as well, cf. \cite{LePrince07}.

\medskip

3. Another interesting issue is the connection of the harmonic measure with
the invariant measures of the geodesic flow (more generally, of Weyl chambers
in the case of higher rank locally symmetric spaces). As it was proved by
Katok and Spatzier \cite{Katok-Spatzier96}, the only invariant measure of the
Weyl chamber flow associated to a lattice $\G$ in $SL(d,\R)$ with positive
entropy along all one-dimensional directions is the Haar measure. In view of
the correspondence between the Radon invariant measures of the Weyl chamber
flow and $\G$-invariant Radon measures on the principal stratum of the product
$\B\times\B$ (cf. \cite{Kaimanovich90} in the hyperbolic case), it implies
that singular harmonic measures on $\B$ either can not be lifted to a Radon
invariant measure on $\B\times\B$ (contrary to the hyperbolic case
\cite{Kaimanovich90}) or, if they can be lifted to a Radon invariant measure
on $\B\times\B$, then the associated invariant measure of the Weyl chamber
flow has a vanishing directional entropy (in our setup the latter option most
likely can be eliminated).

\medskip

The paper has the following structure. In \secref{sec:RW} we set up the
notations and introduce the necessary background information about random
products (random walks) on groups. In \secref{sec:exponents},
\secref{sec:flags} and \secref{sec:harm} we specialize these notions to the
case of matrix random products, remind the Oseledets multiplicative ergodic
theorem, give several equivalent descriptions of the limit flags (Lyapunov
filtrations) of random products, and introduce the harmonic measure $\nu$ on
the flag space $\B$ and its projections $\nu_i$ to the Grassmannians $\Gr_i$
in $\R^d$. In \secref{sec:entr} we remind the definitions of the asymptotic
entropy of a random walk and of the differential entropy of a $\mu$-stationary
measure. In \lemref{lem:dentr} we establish an inequality relating the
exponential rate of decay of translates of a stationary measure along the
sample paths of the reversed random walk with the differential entropy. In
\secref{sec:Hausd} we discuss the notion of dimension of a probability measure
on a compact metric space. We introduce new \emph{lower} and \emph{upper mean
dimensions} and establish in \thmref{th:dims} inequalities (used in the proof
of \thmref{th:dim}) relating them to other more traditional definitions of
dimension (including the classical Hausdorff dimension). After discussing the
notion of the limit set in the higher rank case in \secref{sec:limit}, we
finally pass to proving our Main Theorem in \secref{sec:link} and
\secref{sec:const}. Namely, in \secref{sec:link} we establish the inequality
of \thmref{th:dim}, and in \secref{sec:const} we prove \thmref{th:estimate} on
existence of jump distributions with uniformly bounded entropy and arbitrarily
big Lyapunov exponents.

\section{Random products} \label{sec:RW}

We begin by recalling the basic definitions from the theory of random walks on
discrete groups, e.g., see \cite{Kaimanovich-Vershik83,Kaimanovich00a}. The
\emph{random walk} determined by a probability measure $\mu$ on a countable
group $\G$ is a Markov chain with the transition probabilities
$$
p(g,gh)=\mu(h) \;, \qquad g,h\in\G \;.
$$

\conv{Without loss of generality (passing, if necessary, to the subgroup
generated by the support $\supp\mu$ of the measure $\mu$) we shall always
assume that the measure $\mu$ is \emph{non-degenerate} in the sense that the
smallest subgroup of $\G$ containing $\supp\mu$ is $\G$ itself.}

If the position of the random walk at time $0$ is a point $x_0\in\G$, then its
position at time $n>0$ is the product
$$
x_n = x_0 h_1 h_2 \dots h_n \;,
$$
where $(h_n)_{n\ge 1}$ is the sequence of independent $\mu$-distributed
\emph{increments} of the random walk. Therefore, provided that $x_0$ is the
group identity $e$, the distribution of $x_n$ is the $n$-fold convolution
$\mu^{*n}$ of the measure $\mu$.

Below it will be convenient to consider bilateral sequences of Bernoulli
$\mu$-distributed increments $\h=(h_n)_{n\in\Z}$ and the associated
\emph{bilateral sample paths} $\x=(x_n)_{n\in\Z}$ obtained by extending the
formula
$$
x_{n+1}=x_n h_{n+1}
$$
to all $n\in\Z$ under the condition $x_0=e$, so that
\begin{equation} \label{eq:iso}
x_n =
\begin{cases}
    h_0^{-1}h_{-1}^{-1} \dots h_{n+1}^{-1}\;, & n<0\;; \\
    e\;, & n=0\;; \\
    h_1 h_2\dots h_n\;, & n>0\;. \\
\end{cases}
\end{equation}
Thus, the ``negative part'' $\ch x_n=x_{-n},\;n\ge 0,$ of the bilateral random
walk is the random walk on $\G$ starting from the group identity at time 0 and
governed by the \emph{reflected measure} $\ch\mu(g)=\mu(g^{-1})$.

We shall denote by $\P$ the probability measure in the space $\X=\G^\Z$ of
bilateral sample paths $\x$ which is the image of the Bernoulli measure in the
space of bilateral increments $\h$ under the isomorphism \eqref{eq:iso}. It is
preserved by the transformation $T$ induced by the shift in the space of
increments:
\begin{equation} \label{eq:T}
(T\x)_n = h_1^{-1} x_{n+1}\;, \qquad n\in\Z\;.
\end{equation}

\section{Lyapunov exponents} \label{sec:exponents}

We shall fix a standard basis $\E=(e_i)_{1\le i\le d}$ in $\R^d$ and identify
elements of the Lie group $G=SL(d,\R)$ with their matrices in this basis.
Throughout the paper we shall use the standard Euclidean norms associated with
this basis both for vectors and matrices in $\R^d$.

\conv{From now on we shall assume that $\G$ is a countable subgroup of the
group $G=SL(d,\R)$.}

Denote by $A$ the \emph{Cartan subgroup} of $G$ consisting of diagonal
matrices $a=\diag(a_i)$ with positive entries $a_i$, and by
$$
\af = \left\{ \a=(\a_1,\a_2,\dots,\a_d)\in\R^d : \sum\a_i=0 \right\}
$$
the associated \emph{Cartan subalgebra}, so that $A=\exp\af$. Let
$$
\af_+ = \{\a\in\af: \a_1>\a_2> \dots> \a_d \} \;,
$$
be the standard \emph{positive Weyl chamber} in $\af$, and let
$$
A_+ = \exp\af_+ = \{a\in A: a_1> a_2> \dots > a_d\} \;.
$$
Denote the closures of $\af_+$ and $A_+$ by $\ov{\af_+}$ and $\ov{A_+}$,
respectively.

Any element $g\in G$ can be presented as $g=k_1 a k_2$ with $k_{1,2}\in
K=SO(d)$ (which is a maximal compact subgroup in $G$) and a uniquely
determined $a=a(g)\in\ov{A_+}$ (\emph{Cartan} or \emph{polar} decomposition).

\conv{We shall always assume that the probability measure $\mu$ on $\G$ has a
\emph{finite first moment} in the ambient group $G$, i.e., $\sum \log \|g\|
\mu(g) < \infty$. }

Then the asymptotic behaviour of the random walk $(\G,\mu)$ is described by
the famous \emph{Oseledets multiplicative ergodic theorem} which we shall
state in the form due to Kaimanovich \cite{Kaimanovich89} (and in the
generality suitable for our purposes):

\begin{thm} \label{th:osel}
There exists a vector $\l=\l(\mu)\in\ov\af_+$ (the \emph{Lyapunov spectrum} of
the random walk) such that
$$
\frac1n \log a(x_n) \toto_{n\to+\infty} \l
$$
for $\P$-a.e.\ sample path $\x\in\X$. Moreover, for $\P$-a.e.\ $\x$ there
exists a uniquely determined positive definite symmetric matrix
$$
g=g(\x)=k(\exp\l) k^{-1} \;, \qquad k\in K \;,
$$
such that
$$
\log \|g^{-n}x_n\|=o(n) \quad \textrm{as} \quad n\to+\infty\;.
$$
\end{thm}

\begin{thm}[\cite{Guivarch-Raugi85,Goldsheid-Margulis89}] \label{th:simple}
If, in addition, the group $\G$ is Zariski dense in $G$, then the Lyapunov
spectrum $\l(\mu)$ is \emph{simple}, i.e., it belongs to the Weyl chamber
$\af_+$.
\end{thm}

\begin{rem} \label{rem:inverse}
The Lyapunov spectra $(\l_1,\dots,\l_d)$ and $(\ch\l_1,\dots,\ch\l_d)$ of the
measure $\mu$ and of the reflected measure $\ch\mu$, respectively, are
connected by the formula
$$
\ch\l_i = -\l_{d+1-i} \;, \qquad 1\le i\le d\;.
$$
\end{rem}

\section{Limit flags} \label{sec:flags}

Let $S=G/K$ be the \emph{Riemannian symmetric space} associated with the group
$G$ (e.g., see \cite{Eberlein96} for the basic notions). We shall fix a
reference point $o=K\in S$ (its choice is equivalent to choosing a Euclidean
structure on $\R^d$ such that its rotation group is $K$). Being non-positively
curved and simply connected, the space $S$ has a natural \emph{visibility
compactification} $\ov S = S \cup \pt S$ whose boundary $\pt S$ consists of
asymptotic equivalence classes of geodesic rays in $S$ and can be identified
with the unit sphere of the tangent space at the point $o$ (since any
equivalence class contains a unique ray issued from $o$). The action of the
group $G$ extends from $S$ to $\pt S$, and the orbits of the latter action are
naturally parameterized by unit length vectors $\a\in\ov{\af_+}$: the orbit
$\pt S_\a$ consists of the equivalence classes of all the rays
$\g(t)=k\exp(t\a)o,\; k\in K$. Algebraically the orbits $\pt S_\a$
corresponding to the interior vectors $\a\in\af_+$ are isomorphic to the space
$\B$ of \emph{full flags}
$$
\V=\{V_i\}\;,\quad V_0=\{0\}\subset V_1\subset\dots\subset V_{d-1}\subset
V_d=\R^d
$$
in $\R^d$ (also known as the \emph{Furstenberg boundary} of the symmetric
space $S$), whereas the orbits corresponding to wall vectors are isomorphic to
quotients of $\B$, i.e., to flag varieties for which certain intermediate
dimensions are missing, see \cite{Kaimanovich89}.

\thmref{th:osel} implies that for $\P$-a.e.\ sample path $\x$
$$
d(x_n o, \g(t\|\l\|))=o(n) \quad \textrm{as} \quad n\to+\infty \;,
$$
where $\g$ is the geodesic ray $\g(t)=k\exp\left(t\frac{\l}{\|\l\|}\right)o$,
hence the sequence $x_n o$ converges in the visibility compactification to a
limit point $\bnd\x\in\pt S_{\l/\|\l\|}$. Moreover, $\pt S_{\l/\|\l\|}\cong\B$
by \thmref{th:simple}, so that below we shall consider the aforementioned
\emph{boundary map} $\bnd$ as a map from the path space to the flag space
$\B$.

\thmref{th:osel} and \thmref{th:simple} easily imply the following
descriptions of the limit flag $\bnd\x$.

\begin{prop} \label{pr:lyap}
\hfill
\begin{itemize}
\item[(i)]
Denote by $\V_0$ the \emph{standard flag}
$$
\qquad \{0\} \subset \spn\{e_1\} \subset \spn\{e_1,e_2\} \subset \dots \subset
\spn\{e_1,e_2,\dots,e_{d-1}\} \subset \R^d
$$
associated with the basis $\E$. Then
$$
\bnd\x = k\V_0
$$
for $k=k(\x)\in K$ from \thmref{th:osel}.

\item[(ii)]
The spaces $V_i$ from the flag $\bnd\x$ are increasing direct sums of the
eigenspaces of the matrix $g=g(\x)$ from \thmref{th:osel} taken in the order
of decreasing eigenvalues.

\item[(iii)]
The flag $\bnd\x$ is the \emph{Lyapunov flag} of the sequence $x_n^{-1}$,
i.e.,
$$
\lim_{n\to+\infty} \frac1n \log \|x_n^{-1} v\| = -\l_i \qquad \forall\, v\in
V_i\setminus V_{i-1} \;.
$$

\item[(iv)]
For any smooth probability measure $\th$ on $\B$ and $\P$-a.e.\ sample path
$\x$
$$
\lim_{n\to+\infty} x_n\th = \d_{\bnd\x}
$$
in the weak$^*$ topology of the space of probability measures on $\B$.
\end{itemize}
\end{prop}

\begin{rem}
Equivariance of the boundary map $\bnd$ implies that for the transformation
$T$ \eqref{eq:T}
$$
\bnd T^n\x = x_n^{-1}\bnd\x \qquad\forall\,n\in\Z \;.
$$
In particular,
\begin{equation} \label{eq:Tc}
\bnd T^{-n}\x = \ch x_n^{-1}\bnd\x \qquad\forall\,n\ge 0 \;.
\end{equation}
\end{rem}

\section{Harmonic measure} \label{sec:harm}

\begin{defn}
The image $\nu=\bnd(\P)$ of the probability measure $\P$ in the path space
$\X$ under the map $\bnd:\X\to\B$ is called the \emph{harmonic measure} of the
random walk $(\G,\mu)$. In other words, $\nu$ is the distribution of the limit
flag $\bnd\x$ under the measure $\P$.
\end{defn}

The harmonic measure is \emph{$\mu$-stationary} in the sense that it is
invariant with respect to the convolution with $\mu$:
$$
\mu*\nu = \sum_g \mu(g) g\nu = \nu \;.
$$

\begin{thm}[\cite{Guivarch-Raugi85,Goldsheid-Margulis89}] \label{th:sub}
Under conditions of \thmref{th:simple} $\nu$ is the unique $\mu$-stationary
probability measure on $\B$, and any proper algebraic subvariety of $\B$ is
$\nu$-negligible.
\end{thm}

\begin{thm}[\cite{Kaimanovich85,Ledrappier85}]
Under conditions of \thmref{th:simple}, if the subgroup $\G$ is discrete, then
the measure space $(\B,\nu)$ is isomorphic to the Poisson boundary of the
random walk $(\G,\mu)$.
\end{thm}

\begin{rem}
If $\G$ is not discrete, then the random walk may have limit behaviours other
than described just by the limit flags, or, in other words, the Poisson
boundary of the random walk $(\G,\mu)$ may be bigger than the flag space, see
\cite{Kaimanovich-Vershik83} for the first example of this kind (the
dyadic-rational affine group) and \cite{Bader-Shalom06,Brofferio06} for the
recent developments.
\end{rem}

Denote by $\Gr_i$ the dimension $i$ \emph{Grassmannian} (the space of all
dimension $i$ subspaces) in $\R^d$. There is a natural projection
$\pi_i:\B\to\Gr_i$ which consists in assigning to any flag in $\R^d$ its
dimension $i$ subspace. Let $\nu_i=\pi_i(\nu),\;1\le i\le d-1,$ be the
associated images of the measure $\nu$. Obviously, the measures $\nu_i$ are
$\mu$-stationary (along with the measure $\nu$). We shall also use the
notation
$$
\bnd_i\x = \pi_i(\bnd\x) \in \Gr_i \;,
$$
so that $\nu_i=\bnd_i(\P)$.

We shall embed each Grassmannian $\Gr_i$ into the projective space
$P\bigwedge^i \R^d$ in the usual way and define a $K$-invariant metric
$\r=\r_i$ on the latter as
\begin{equation} \label{eq:metric}
\r (\xi,\z) = \sin\an(\xi,\z) \;,
\end{equation}
where the angle (varying between $0$ and $\pi/2$) is measured with respect to
the standard Euclidean structure on $\bigwedge^i\R^d$ determined by the basis
$\E$ (so that $(e_{j_1}\wedge \cdots \wedge e_{j_i})_{1\le j_1 < \cdots < j_i
\le d}$ is an orthonormal basis of $\bigwedge^i\R^d$).

The \emph{Furstenberg formula} (see \cite{Furstenberg63a,Bougerol-Lacroix85})
\begin{equation} \label{eq:furst}
\l_1 + \l_2+\dots+\l_i =
 \sum_g \mu(g) \int_{\Gr_i} \log\frac{\|gv\|}{\|v\|}\ d\nu_i (\ov{v}) \;,
 \qquad 1\le i\le d-1 \;,
\end{equation}
where $v\in\bigwedge^i\R^d$ is the vector presenting a point
$\ov{v}\in\Gr_i\subset P\bigwedge^i\R^d$, relates Lyapunov exponents with the
harmonic measure.

\section{Entropy} \label{sec:entr}

Recall that if the measure $\mu$ has a finite entropy $H(\mu)$, then the
\emph{asymptotic entropy} of the random walk $(\G,\mu)$ is defined as
\begin{equation}\label{eq:entr}
h(\G,\mu)=\lim_{n\to+\infty} \frac{H(\mu^{*n})}n \le H(\mu)\;,
\end{equation}
where $H(\cdot)$ denotes the usual entropy of a discrete probability measure.
The asymptotic entropy can also be defined ``pointwise'' along sample paths of
the random walk as
$$
h(\G,\mu) = \lim_{n\to+\infty} -\frac1n\log\mu^{*n}(x_n) \;,
$$
where the convergence holds both $\P$-a.e.\ and in the space $L^1(\X,\P)$, see
\cite{Kaimanovich-Vershik83,Derriennic86}.

The \emph{$\mu$-entropy} (\emph{Furstenberg entropy, differential entropy}) of
a $\mu$-stationary measure $\th$ on a $\G$-space $X$ is defined as
\begin{equation}\label{eq:def diff}
     E_\mu(X,\th) = -\sum_{g\in \G}\mu(g)\int\log \frac{dg^{-1}\th}{d\th}(b)d\th(b)
     \;,
\end{equation}
and it satisfies the inequality $E_\mu(X,\th)\le h(\G,\mu)$, see
\cite{Furstenberg71,Kaimanovich-Vershik83,Nevo-Zimmer02}.

\conv{We shall always assume that the probability measure $\mu$ on $\G$ has a
\emph{finite entropy} $H(\mu)<\infty$. }

In our context, if the subgroup $\G$ is discrete in $SL(d,\R)$, then
finiteness of the first moment of the measure $\mu$ easily implies that
$H(\mu)<\infty$ (e.g., see \cite{Derriennic86}). Therefore,
$E_i=E_\mu(\Gr_i,\nu_i)<\infty$.

Below we shall need the following routine estimate (in fact valid for an
arbitrary quotient of the Poisson boundary).

\begin{lem} \label{lem:dentr}
For any index $i\in\{1,2,\dots,d-1\}$, any subset $A\subset\Gr_i$ with
$\nu_i(A)>0$ and $\P$-a.e.\ sample path $\x$
\begin{equation}\label{eq:dentr}
\limsup_{n\to+\infty}\left[-\frac{1}{n} \log \ch x_n\nu_i(A)\right] \le E_i
\;.
\end{equation}
\end{lem}

\begin{proof}
Put
$$
F_i(\x)=-\log\frac{d\ch x_1\nu_i}{d\nu_i}
(\bnd_i\x)=-\log\frac{d\nu_i(\bnd_iT^{-1}\x)}{d\nu_i(\bnd_i\x)}
$$
(see formula \eqref{eq:Tc}), so that
$$
E_i = \int F_i(\x)\,d\P(\x) \;.
$$
Then
$$
\begin{aligned}
-\log\frac{d\ch x_n\nu_i}{d\nu_i} (\bnd_i\x)
 &= -\log\frac{d\nu_i(\bnd_iT^{-n}\x)}{d\nu_i(\bnd_i\x)} \\
 &= F_i(\x) + F_i(T^{-1}\x) + \dots + F_i(T^{-n+1}\x) \;,
\end{aligned}
$$
whence by the ergodic theorem
$$
-\frac1n \log\frac{d\ch x_n\nu_i}{d\nu_i} (\bnd_i\x) \toto_{n\to+\infty} E_i
$$
in $L^1(\X,\P)$. It implies that
$$
-\frac1n \int_A \log\frac{d\ch x_n\nu_i}{d\nu_i} (\xi)
\frac{d\nu_i(\xi)}{\nu_i(A)} \toto_{n\to+\infty} E_i \;,
$$
which, by a convexity argument, yields the claim.
\end{proof}

\section{Dimension of measures} \label{sec:Hausd}

Let us recall several notions of the dimension of a probability measure $m$ on
a compact metric space $(Z,\r)$ (all the details, unless otherwise specified,
can be found in the book \cite{Pesin97}). These notions roughly fall into two
categories: the \emph{global} ones are obtained by looking at the dimension of
sets which ``almost'' coincide with $Z$ (up to a piece of small measure $m$),
whereas the \emph{local} ones are related to the asymptotic behavior of the
ratios $\log m B(z,r)/\log r$ as the radius $r$ tends to $0$.

\subsection{Global definitions}

The \emph{Hausdorff dimension} of the measure $m$ is
$$
\dim_H m = \inf\left\{ \dim_H A : m(A)=1 \right\} \;,
$$
where $\dim_H A$ denotes the Hausdorff dimension of a subset $A\subset Z$.

The \emph{lower} and the \emph{upper box dimensions} of a subset $A\subset Z$
are defined, respectively, as
$$
\un\dim_B A = \liminf_{r\to 0}\frac{\log N(A,r)}{\log 1/r}
 \quad\textrm{and}\quad
 \ov{\dim}_B A = \limsup_{r\to 0}\frac{\log N(A,r)}{\log 1/r} \;,
$$
where $N(A,r)$ is the minimal number of balls of radius $r$ needed to cover
$A$. Ledrappier \cite{Ledrappier81} also considered the minimal number
$N(r,\e,m)$ of balls of radius $r$ such that the measure $m$ of their union is
at least $1-\e$ and defined the ``fractional dimension'' of the measure $m$ as
$$
\ov\dim_L m = \sup_{\e\to 0} \limsup_{r\to 0} \frac{\log N(r,\e,m)}{\log 1/r}
$$
(we use the notation from \cite{Young82} and below shall call it the
\emph{upper Ledrappier dimension}). As it was noticed by Young \cite{Young82},
in the same way one can also define what we call the \emph{lower Ledrappier
dimension} of the measure $m$
$$
\un\dim_L m = \sup_{\e\to 0} \liminf_{r\to 0} \frac{\log N(r,\e,m)}{\log 1/r}
$$
as well as, in modern terminology, its \emph{lower} and the \emph{upper box
dimensions}, respectively,
$$
\un\dim_B m = \liminf_{m(A)\to 1} \left\{ \un\dim_B A \right\}  \quad
\textrm{and} \quad
 \ov\dim_B m = \liminf_{m(A)\to 1} \left\{ \ov\dim_B A \right\} \;.
$$
Obviously,
$$
\un\dim_L m \le \ov\dim_L m \quad\textrm{and}\quad \un\dim_B m \le \ov\dim_B m
\;.
$$

The difference between the Ledrappier and the box dimensions is that in the
definition of the box dimensions it is the same set $A$ which has to be
covered by balls with varying radii $r$, unlike in the definition of
Ledrappier, so that
$$
\un\dim_L m \le \un\dim_B m \quad\textrm{and}\quad \ov\dim_L m \le \ov\dim_B m
\;.
$$

By \cite[Proposition 4.1]{Young82},
\begin{equation}\label{eq:dims}
\dim_H m \le \un\dim_L m  \;.
\end{equation}

\subsection{Local definitions}

The \emph{lower} and the \emph{upper pointwise dimensions} of the measure $m$
at a point $z\in Z$ are
$$
\un\dim_P m (z) = \liminf_{r\to 0}\frac{\log m B(z,r)}{\log r}
 \quad
 \textrm{and}
 \quad
 \ov\dim_P m (z) = \limsup_{r\to 0}\frac{\log m B(z,r)}{\log r} \;,
$$
respectively. Then
\begin{equation} \label{eq:HD}
\dim_H m = \esssup_z \un\dim_P m (z) \;.
\end{equation}
In particular, if $m$-a.e.\
\begin{equation} \label{eq:D}
\lim_{r\to 0}\frac{\log m B(z,r)}{\log r}=D \;,
\end{equation}
then $\dim_H m=D$. Moreover, in this case all the reasonable definitions of
dimension of the measure $m$ give the same result \cite[Theorem 4.4]{Young82}.

\begin{defn} \label{def:M}
In the situation when the convergence in \eqref{eq:D} holds just in
probability we shall say that $D$ is the \emph{mean dimension} $\dim_M m$ of
the measure $m$. We shall also introduce the \emph{lower} and the \emph{upper
mean dimensions} of the measure $m$ as, respectively,
$$
\begin{aligned}
\un\dim_M m &= \sup\left\{t: \left[\frac{\log m B(z,r)}{\log r} -
t\right]_-\toto^m 0 \right\} \;, \\
\ov\dim_M m &= \inf\left\{t: \left[\frac{\log m B(z,r)}{\log r} -
t\right]_+\toto^m 0 \right\} \;,
\end{aligned}
$$
where $[t]_+=\max\{0,t\}, [t]_-=\min\{0,t\}$, and $\displaystyle{\toto^m}$
denotes convergence in probability with respect to the measure $m$.
\end{defn}

The definition of $\dim_M$ first appeared in Ledrappier's paper
\cite{Ledrappier83} (also see \cite{Ledrappier84}), whereas $\un\dim_M$ and
$\ov\dim_M$ (although obvious generalizations of $\dim_M$) are, apparently,
new. Clearly, $\un\dim_M m \le \ov\dim_M m$. In slightly different terms,
$[\un\dim_M m, \ov\dim_M m]$ is the minimal closed subinterval of $\R$ with
the property that for any closed subset $I$ of its complement
$$
m\left\{z\in Z: \frac{\log m B(z,r)}{\log r}\in I\right\} \toto_{r\to 0} 0 \;.
$$
In particular, if $\dim_M m$ exists, then $\un\dim_M m=\ov\dim_M m=\dim_M m$.

\subsection{Mean, box and Ledrappier dimensions} \label{sec:bes}

We shall now establish simple inequalities between these dimensions.

\begin{prop}
For any probability measure $m$ on a compact metric space $(Z,\r)$
$$
\un\dim_M m \le \un\dim_L m \;.
$$
\end{prop}

\begin{proof}
Fix a number $D<\un\dim_M m$. Then for any $\e>0$ there exist $r_0>0$ and a
set $A\subset Z$ with $m(A)>1-\e$ and such that $m B(z,r)\le r^D$ for all
$z\in A$ and $r\le r_0$. Suppose now that
$$
m \left( \bigcup_i B(z_i,r) \right) \ge 1-\e
$$
for a certain set of points $\{z_i\}$ of cardinality $N$ and a certain $r\le
r_0/2$. If $B(z_i,r)$ intersects $A$, then $B(z_i,r)\subset B(z,2r)$ for some
$z\in A$, so that
$$
m\left( A \cap B(z_i,r) \right) \le m B(z_i,r) \le m B(z,2r) \le (2r)^D \;.
$$
Thus,
$$
1-2\e \le m \left( A \cap \bigcup_i B(z_i,r) \right) \le N (2r)^D \;,
$$
whence the claim.
\end{proof}

For establishing the inverse inequality we shall additionally require that the
metric space $(Z,\r)$ has the \emph{Besicovitch covering property}, i.e., that
for any precompact subset $A\subset Z$ and any bounded function $r:A\to\R_+$
(important particular case: $r$ is constant) the cover $\{B(z,r(z)), z\in A\}$
of $A$ contains a countable subcover whose multiplicity is bounded from above
by a universal constant $M=M(Z,\r)$ (recall that the \emph{multiplicity} of a
cover is the maximal number of elements of this cover to which a single point
may belong). The Besicovitch property is, for instance, satisfied for the
Euclidean space, hence, for all its compact subsets endowed with a metric
which is Lipschitz equivalent to the Euclidean one. Therefore, it is satisfied
for each of the Grassmannians $\Gr_i$ endowed with the metric
\eqref{eq:metric}.

\begin{prop}
For any probability measure $m$ on a compact space $(Z,\r)$ satisfying the
Besicovitch property
$$
\ov\dim_B m \le \ov\dim_M m \;.
$$
\end{prop}

\begin{proof}
Take a number $D>\ov\dim_M m$, let
$$
A_r = \left\{z\in Z : m B(z,r)\ge r^D \right\} \;,
$$
and consider a cover of $A_r$ by the balls $B(z_i,r),\;z_i\in A_r$ obtained
from applying the Besicovitch property. The cardinality of this cover is at
least $N(A_r,r)$, whereas its multiplicity is at most $M$, whence
$$
N(A,r)r^D \le \sum_i m B(z_i,r) \le M \;,
$$
so that
$$
\frac{\log N(A_r,r)}{\log 1/r}\le D+\frac{\log M}{\log 1/r} \;.
$$
For $r\to0$ the right-hand side of the above inequality tends to $D$, whereas
$m(A_r)\to 1$ by the choice of $D$, whence $\ov\dim_B m\le D$.
\end{proof}

\subsection{Final conclusions}

Taking stock of the above discussion we obtain

\begin{thm}\label{th:dims}
For any probability measure $m$ on a compact metric space
$$
\left.  \begin{aligned} &\dim_H m\\ &\,\un\dim_M m\\ \end{aligned}  \right\}
\le \un\dim_L m \le \un\dim_B m \;.
$$
If, in addition, the space has the Besicovitch property, then
$$
\ov\dim_L m \le \ov\dim_B m \le \ov\dim_M m \;.
$$
\end{thm}

\begin{rem}
There is no general inequality between $\dim_H m$ and $\un\dim_M m$. For
instance, take two singular measures $m_1,m_2$ for which the dimensions
\eqref{eq:D} exist and are different, say, $D_1<D_2$, and let $m$ be their
convex combination. Then $\un\dim_M m=D_1$, whereas $\dim_H m=D_2>\un\dim_M m$
by \eqref{eq:HD}.

On the other hand, by exploiting the difference between the convergence in
probability and the convergence almost everywhere one can also construct
examples with $\dim_H m<\un\dim_M m$. We shall briefly describe one such
example. Let $Z'$ be the space of unilateral binary sequences
$\a=(\a_1,\a_2,\dots)$ with the uniform measure $m'$ and the usual metric
$\r'$ for which $-\log\r'(\a,\b)$ is the length of the initial common segment
of the sequences $\a$ and $\b$. Take a sequence of cylinder sets $A_n$ with
the property that $m' A_n\to 0$, but any point of $Y$ belongs to infinitely
many sets $A_n$. Also take a sequence of integers $s_n$ (to be specified
later), and let $Z$ be the image of the space $Z'$ under the following map:
given a sequence $\a\in Z'$ take the set $I=\{n:\a\in A_n\}$ and replace with
0 all the symbols $\a_k$ with $s_n\le k\le 2s_n$ for a certain $n\in I$. The
space $Z$ is endowed with the quotient measure $m$ and the quotient metric
$\r$. If the sequence $s_n$ is very rapidly increasing, then \eqref{eq:HD} can
be used to show that $\dim_H m = \log 2/2$, whereas $\un\dim_M m=\dim_M m=\log
2$.
\end{rem}

\section{Limit set} \label{sec:limit}

Denote by  $\Pr(\B)$ the compact space of probability measures on the flag
space $\B$ endowed with the weak$^*$ topology, and let $m$ be the unique
$K$-invariant probability measure on $\B$. Then the map $go\mapsto gm$
determines an embedding of the symmetric space $S=G/K$ into $\Pr(\B)$, and
gives rise to the \emph{Satake--Furstenberg compactification} of $S$. Its
boundary $\ov S\setminus S$ contains the space $\B$ under the identification
of its points with the associated delta-measures (but, unless the rank of
$G=SL(d,\R)$ is 1, i.e., $d=2$, it also contains other limit measures, see
\cite{Guivarch-Ji-Taylor98}). The \emph{limit set} $L_\G$ of a subgroup
$\G\subset G$ in this compactification is then defined (see \cite{Guivarch90})
as
$$
L_\G =\ov{\G o}\cap \B \subset \B \;.
$$
The limit set is obviously $\G$-invariant and closed. Moreover,

\begin{thm}[\cite{Guivarch90}] \label{th:limset}
If the group $\G$ is Zariski dense in $G$, then its action on the limit set
$L_\G$ is minimal (i.e., $L_\G$ has no proper closed $\G$-invariant subsets).
\end{thm}

Below we shall need the following elementary property.

\begin{prop} \label{pr:min}
Under the conditions of \thmref{th:limset} let $U\subset\B$ be an open set
with $U\cap L_\G \neq \emp$. Then there exists finitely many elements
$\g_1,\dots,\g_r\in\G$ such that
$$
L_\G \subset \bigcup_i \g_i U\ .
$$
\end{prop}

\begin{proof}
Minimality of $L_\G$ means that any closed $\G$-invariant subset of $\B$
either contains $L_\G$ or does not intersect it. Since the set
$\B\setminus\bigcup_{\g\in\G}\g U$ does not contain $L_\G$ (as $U\cap L_\G
\neq \emp$), it does not intersect $L_\G$, so that $L_\G \subset
\bigcup_{\g\in\G}\g U$. Finally, since $L_\G$ is compact, the above cover
contains a finite subcover.
\end{proof}

\begin{rem}
\thmref{th:limset} and \propref{pr:min} obviously carry over to the
projections $L^i_\G=\pi_i(L_\G)\subset\Gr_i$ of the limit set $L_\G$ to the
Grassmannians $\Gr_i$.
\end{rem}

By \propref{pr:lyap}(iv) (see \cite{Guivarch-Raugi85} for a more general
argument) a.e.\ sample path $\x$ converges as $n\to+\infty$ to the limit flag
$\bnd\x$ in the Satake--Furstenberg compactification as well. Therefore,
$\supp\nu\subset L_\G$. If $\supp\mu$ generates $\G$ as a semigroup, then
$\mu$-stationarity of the harmonic measure $\nu$ implies its quasi-invariance,
so that in this case $\supp\nu$ is $\G$-invariant, and, if $\G$ is Zariski
dense, $\supp\nu=L_\G$ by \thmref{th:limset}.

\begin{rem}
If $\G$ is a lattice, then $L_\G=\B$. On the other hand, if $\G$ is Zariski
dense, then, as it is shown in \cite{Guivarch90}, $L_\G$ has positive
Hausdorff dimension, which is deduced from the positivity of the dimension of
the harmonic measure (under the assumption that $\mu$ has an exponential
moment). Also see \cite{Link04} for recent results on the Hausdorff dimension
of the \emph{radial limit set}. It would be interesting to investigate
existence of random walks such that their harmonic measure has the maximal
Hausdorff dimension (equal to the dimension of the limit set), cf.
\cite{Connell-Muchnik07}.
\end{rem}

\section{Dimension of the harmonic measure} \label{sec:link}

\subsection{Rate of contraction estimate}

Recall (see \secref{sec:RW}) that the negative part $(\ch x_n)_{n\ge
0}=(x_{-n})_{n\ge 0}$ of a bilateral sample path $\x$ performs the random walk
on $\G$ governed by the reflected measure $\ch\mu$. Denote by $\bnd^-\x\in\B$
the corresponding limit flag, and for $i\in\{1,2,\dots,d-1\}$ let
$$
\xi_\x=\left(\bnd_{d-i}^-\x\right)^\bot\in\Gr_i
$$
be the orthogonal complement in $\R^d$ of the $(d-i)$-dimensional subspace of
the flag $\bnd^-\x$ (for simplicity we omit the index $i$ in the notation for
$\xi_\x$; the Grassmannian $\Gr_i$ to which it belongs should always be clear
from the context).

\begin{thm} \label{th:contr}
For any Grassmannian $\Gr_i$, any $r<1$, and $\P$-a.e.\ $\x\in\X$,
$$
\liminf_n \left[ - \frac1n \log\diam \ch x_n^{-1} B(\xi_\x,r) \right] \ge \l_i
- \l_{i+1} \;.
$$
\end{thm}

\begin{proof}
Let us consider first the case when $i=1$. Then by the definition of the
metric $\r$ \eqref{eq:metric} on $\Gr_1\cong P\R^d$ the ball $B(\xi_\x,r)$
consists of the projective classes of all the vectors $v+w$, where
$v\in\xi_\x\setminus\{0\}$ and $w\bot\xi_\x$ with
\begin{equation} \label{eq:C}
\frac{\|w\|}{\|v\|}\le C \quad \textrm{for a certain constant} \quad C=C(r)
\;.
\end{equation}
By \propref{pr:lyap}(iii) applied to the random walk $(\G,\ch\mu)$ we have
that
$$
\lim_n \frac1n \log\frac{\|\ch x_n^{-1} w\|}{\|\ch x_n^{-1} v\|} \le\l_2-\l_1
$$
uniformly under condition \eqref{eq:C}, whence the claim.

The general case argument follows the same lines with the only difference that
now instead of the action on $P\R^d$ one has to consider the action on
$P\bigwedge^i\R^d$. The sequence of matrices $\bigwedge^i \ch x_n^{-1}$ (which
are the images of $\ch x_n^{-1}$ in the $i$-th external power representation)
is also Lyapunov regular with the Lyapunov spectrum consisting of all the sums
of the form $\l_{j_1}+\l_{j_2}+\dots+\l_{j_i}$ with $1\le j_1 < \cdots < j_i
\le d$. In particular, the top of the spectrum (corresponding precisely to the
eigenspace $\xi_\x\in P\bigwedge^i\R^d$ as orthogonal to the highest dimension
proper subspace of the Lyapunov flag) is $\l_1+\l_2+\dots+\l_i$, whereas the
second point in the spectrum is $\l_1+\l_2+\dots+\l_{i-1}+\l_{i+1}$, the
spectral gap being $\l_i-\l_{i+1}$.
\end{proof}

\begin{prop}\label{pr:balls}
For any index $i\in\{1,2,\dots,d-1\}$, any $\e>0$ and $\P$-a.e.\ sample path
$\x$ the sets
$$
A_N = \bigcap_{n\ge N} \ch x_n B \left( \bnd_iT^{-n}\x,
e^{-n(\l_i-\l_{i+1}-\e)}\right) \subset \Gr_i
$$
have positive measure $\nu_i$ for all sufficiently large $N$.
\end{prop}

\begin{proof}
By definition of the metric $\r$ \eqref{eq:metric} its values do not exceed 1,
and for any $\xi\in\Gr_i$ the radius 1 sphere $S(\xi,1)$ centered at $\xi$
consists precisely of those $\z\in\Gr_i$ for which the vectors from
$\bigwedge^i\R^d$ associated with $\xi$ and $\z$ are orthogonal. The latter
condition on $\z$ is algebraic, so that by \thmref{th:sub} $\nu_i S(\xi,1)=0$
for any $\xi\in\Gr_i$ (the measure $\nu_i$ being the projection of the measure
$\nu$). Therefore, for a.e.\ sample path $\x$ there exists $r<1$ such that
$\nu_i B(\xi_\x,r)>0$. On the other hand, by \thmref{th:contr} and
\eqref{eq:Tc} $B(\xi_\x,r)\subset A_N$ for all sufficiently large $N$.
\end{proof}

\subsection{Dimension  estimate}

Recall that $\bnd_i T^{-n}\x = \ch x_n^{-1}\bnd_i\x$ \eqref{eq:Tc}. By
\lemref{lem:dentr} and \propref{pr:balls}, for $\P$-a.e.\ sample path $\x$
\begin{equation}\label{eq:minentr}
\begin{split}
&\limsup_{n\to\infty} \left[ -\frac1n \log\nu_i B \left( \bnd_iT^{-n}\x,
e^{-n(\l_{i}-\l_{i+1}-\e)} \right) \right] \\
 &\le \limsup_{n\to\infty}\left[-\frac{1}{n} \log \ch x_n\nu_i(A_N)\right] \le E_i \;.
\end{split}
\end{equation}
The left-hand side of this inequality looks almost like in the definition of
the pointwise dimension, with the only difference that the ball centers vary.
We shall take care of this difference by switching to the mean dimension and
using $\mu$-stationarity of the measures $\nu_i$.

\begin{thm}\label{th:dim}
For any $i\in\{1,2,\dots,d-1\}$,
$$
\dim_H(\nu_i)\le\ov\dim_B(\nu_i)\le\ov\dim_M\nu_i \le
\frac{E_i}{\l_i-\l_{i+1}} \;.
$$
\end{thm}

\begin{proof}
Let
$$
A = \bigcap_\eta\bigcup_N\bigcap_{n\ge N}\Big\{\x :
    -\frac1n\log {\nu}^i B\big(\bnd_iT^{-n}\x,e^{-n(\l_i-\l_{i+1}-\e)}\big) \le
    E_i+\eta\Big\} \;.
$$
be the set of sample paths satisfying condition \eqref{eq:minentr}. Since
$\P(A)=1$, for any $\eta,\chi>0$
$$
\P\Big\{\x :
 -\frac{1}{n}\log{\nu}^i B\big(\bnd_iT^{-n}\x,e^{-n(\l_i-\l_{i+1}-\e)}\big)
                                \le E_i+\eta\Big\}\ge 1-\chi
$$
for all sufficiently large $n$. The transformation $T$ preserves the measure
$\P$, so that its image under the map $\x\mapsto\bnd_iT^{-n}\x$ is $\nu_i$,
whence the rightmost inequality. The other inequalities follow from
\thmref{th:dims} because the metric $\r$ has the Besicovitch property (see
\secref{sec:bes}).
\end{proof}

\section{The construction}\label{sec:const}

For the rest of this Section we shall fix a probability measure $\mu$ on a
subgroup $\G\subset G$ satisfying our standing assumptions from
\secref{sec:RW} and \secref{sec:exponents}.

\conv{In addition we shall assume in this Section that $\supp\mu$ generates
$\G$ as a semigroup.}

Recall that an element of the group $G=SL(d,\R)$ is called \emph{$\R$-regular}
if it is diagonalizable over $\R$ and the absolute values of its eigenvalues
are pairwise distinct. By \cite{Benoist-Labourie93} any Zariski dense subgroup
of $G$ contains such an element. Let us fix an $\R$-regular element $\g\in\G$
and consider the sequence of probability measures on $\G$
\begin{equation} \label{eq:mu_k}
\mu^k = \frac12\mu+\frac14\left( \d_{\g^k}+\d_{\g^{-k}} \right) \;.
\end{equation}
where $\d_g$ denotes the Dirac measure at $g$.

\begin{rem}
Actually, we just need that $\g$ be diagonalizable over $\R$ with the absolute
value of some of its eigenvalues being not $1$.
\end{rem}

Denote by $\nu^k$ the harmonic measures on the flag space $\B$ of the random
walks $\left(\G,\mu^k\right)$, and by $\nu^k_i$ their quotients on the
Grassmannians $\Gr_i,\;1\le i \le d-1$. Our Main Theorem follows from

\begin{thm} \label{th:main}
$$
\min_i\{\dim_H\nu_i^k\} \toto_{k\to\infty} 0 \;.
$$
\end{thm}

In its turn, \thmref{th:main} is an immediate consequence of a combination of
the inequality from \thmref{th:dim} and

\begin{thm} \label{th:estimate}
The measures $\mu^k$ \eqref{eq:mu_k} have the property that the entropies
$H(\mu^k)$ are uniformly bounded, whereas the lengths of their Lyapunov
vectors $\l(\mu^k)$ (equivalently, the top Lyapunov exponents $\l_1(\mu^k)$)
go to infinity.
\end{thm}

The rest of this Section is devoted to a proof of \thmref{th:estimate} split
into a number of separate claims. The first one is obvious:

\begin{claim}\label{pr:h is bdd}
The measures $\mu^k$ have uniformly bounded entropies $H\left(\mu^k\right)$.
\end{claim}

In view of the discussion from \secref{sec:entr} it immediately implies

\begin{cor}\label{cor:unif}
The asymptotic entropies $h\left(\G,\mu^k\right)$ and therefore all the
differential entropies $E^k_i=E_{\mu^k}\left(\Gr_i,\nu_i^k\right)$ are
uniformly bounded.
\end{cor}

For estimating the top Lyapunov exponent $\l_1\left(\mu^k\right)$ we shall use
the Furstenberg formula \eqref{eq:furst}, by which
\begin{equation} \label{eq:l1}
\begin{aligned}
\l_1 \left(\mu^k\right)
 &=  \frac12  \sum_g \mu(g)\int_{P\R^d} \log\frac{\|gv\|}{\|v\|}\
 d\nu^k_1(\ov{v}) \\
  &\qquad + \frac14\int_{P\R^d}
  \left( \log\frac{\|\g^k v\|}{\|v\|}+ \log\frac{\|\g^{-k} v\|}{\|v\|}\right)
  d\nu^k_1(\ov{v}) \;.
\end{aligned}
\end{equation}

The absolute value of the first term of this sum is uniformly bounded because
$\mu$ has a finite first moment. For dealing with the second term of the sum
\eqref{eq:l1} we need the following claims.

\begin{claim}\label{claim:g^k}
The sum
$$
\left( \log\frac{\|\g^k v\|}{\|v\|}
 + \log\frac {\|\g^{-k} v\|} {\|v\|}\right)
$$
is bounded from below uniformly on $v\in\R^d\setminus\{0\}$ and $k\ge 0$.
\end{claim}

\begin{proof}
If $\d$ is a diagonal matrix, then obviously,
$$
\|\d^k v\| \|\d^{-k}v\| \ge \la \d^k v,\d^{-k} v\ra = \|v\|^2 \;.
$$
Now, $\g=h^{-1}\d h$ is diagonalizable, whence
$$
\begin{aligned}
\|\g^k v\| \|\g^{-k}v\| &\ge \|h\|^{-2} \|\d^k h v\| \|\d^{-k} h v \| \ge
\|h\|^{-2} \|hv\|^2\\ &\ge \|h\|^{-2}\|h^{-1}\|^{-2} \|v\|^2 \;.
\end{aligned}
$$
\end{proof}

\begin{claim} \label{claim:opens}
For any open subset $U\subset P\R^d$ which intersects the limit set $L^1_\G$
(see \secref{sec:limit}), the measures $\nu^k_1(U)$ are bounded away from zero
uniformly on $k$.
\end{claim}

\begin{proof}
By \propref{pr:min} there exists a finite set $A\subset\G$ such that
$$
L^1_\G\subset \bigcup_{g\in A} g^{-1} U \;.
$$
Since $\supp\mu$ generates $\G$ as a semigroup, for each $g\in A$ there exists
an integer $s=s(g)$ such that $g\in\supp\mu^{*s}$. Then by
$\mu^k$-stationarity of $\nu_1^k$, for each $g\in A$
$$
\nu^k_1(U) \ge \frac1{2^s} \mu^{*s}(g)g\nu_1^k(U) = \frac1{2^s}
\mu^{*s}(g)\nu_1^k\left(g^{-1} U\right) \ge \e \nu_1^k\left(g^{-1} U\right)
$$
for a certain $\e=\e(A)>0$. Summing up the above inequalities over all $g\in
A$ we obtain
$$
|A| \nu^k_1(U) \ge \e \sum_{g\in A} \nu_1^k\left(g^{-1} U\right) \ge
\e\nu_1^k(L^1_\G) = \e \;,
$$
whence the claim.
\end{proof}

Now we are ready to prove

\begin{claim} \label{claim:l1}
$$
\lim_{k\to\infty}\l_1\left(\mu^k\right)=+\infty \;.
$$
\end{claim}

\begin{proof}
We shall fix a diagonalization $\g=h^{-1}\d h$ with
$\d=\diag(\d_1,\dots,\d_d)$ in such a way that $|\d_1|>1$ and $|\d_d|<1$, and
define the open set $U$ as
$$
U = \left\{\ov{v}\in P\R^d :
      \frac{\la h v,e_1\ra }{\|h v\|}>\b
\ \textrm{and}\ \frac{\la h v,e_d\ra }{\|h v\|}>\b \right\} \;,
$$
i.e., by requiring that the first and the last coordinates (with respect to
the standard basis $\E$) of the normalized vector $hv$ be greater than $\b$.
The value of $\b$ is chosen to make sure that $U$ is non-empty (for instance,
one can take $\b=1/2$).

If $\ov{v}\in U$, then
$$
\|\g^k v\|= \|h^{-1}\d^k h v\|\ge \|h\|^{-1} |\d_1|^k \la h v,e_1\ra
    \ge \|h\|^{-1} |\d_1|^k \b \|h v\| \;,
$$
and thus
$$
\frac{\|\g^k v\|}{\| v\|}\ge \|h^{-1}\|^{-1}~\|h\|^{-1}\b |\d_1|^k \;.
$$
In the same way,
$$
\frac{\|\g^{-k} v\|}{\| v\|}\ge \|h^{-1}\|^{-1}~\|h\|^{-1}\b |\d_d|^{-k} \;,
$$
so that
$$
\log\frac{\|\g^k v\|}{\| v\|} + \log\frac{\|\g^{-k} v\|}{\| v\|} \to \infty
$$
as $k\to\infty$ uniformly on $\ov{v}\in U$, which in view of \eqref{eq:l1} in
combination with \claimref{claim:g^k} and \claimref{claim:opens} finishes the
argument.
\end{proof}

\begin{rem}
The measure $\mu$ in our construction can clearly be chosen symmetric and, if
the group $\G$ is finitely generated, finitely supported. Obviously, the
measures $\mu^k$ then also have these properties, so that \emph{singular
harmonic measures can be produced by symmetric finitely supported measures on
the group}.
\end{rem}

\bibliographystyle{amsalpha}
\bibliography{C:/Sorted/MyTEX/mine}

\end{document}